\documentclass[a4paper, 11pt]{article}

\usepackage[
            includefoot,  
            marginparwidth=0in,     
            marginparsep=0in,       
            margin=1.45in,               
            includemp]{geometry}

\usepackage{bbm}
\usepackage{float}
\usepackage{graphicx}                  
\usepackage{amssymb}
\usepackage{amsfonts}
\RequirePackage{amsmath}
\RequirePackage{amsthm}

\usepackage{color}

\usepackage{fancyhdr}

\newtheorem{thm}{Theorem}[section]
\newtheorem{defn}[thm]{Definition}
\newtheorem{prop}[thm]{Proposition}

\newtheorem{rque}{Remark} [section]
\newtheorem{cor}[thm]{Corollary}

\newtheorem*{assumption}{Assumption}

\DeclareMathOperator{\Hess}{Hess}

\DeclareMathOperator{\Var}{Var}

\DeclareMathOperator{\id}{Id}

\newcommand{\R}{\mathbb{R}}

\newcommand{\E}{\mathbb{E}}

\begin{document}

\title{Existence of Stein Kernels under a Spectral Gap, and Discrepancy Bounds}
\date{\today}

\author{Thomas~A.~Courtade$^*$, Max Fathi$^{\dagger}$ and Ashwin Pananjady$^*$\\
~\\
$^{*}$UC Berkeley, Department of Electrical Engineering and Computer Sciences\\
$^{\dagger}$CNRS \& Universit\'e Paul Sabatier, Institut de Math\'ematiques de Toulouse\\
}

\maketitle

\begin{abstract}
We establish existence of Stein kernels for probability measures on $\R^d$ satisfying a Poincar\'e inequality, and obtain bounds on the Stein discrepancy of such measures. Applications to quantitative central limit theorems are discussed, including a new CLT in Wasserstein distance $W_2$ with optimal rate and dependence on the dimension. As a byproduct, we  obtain a stable version of an estimate of the Poincar\'e constant of probability measures under a second moment constraint.  The results extend more generally to the setting of converse weighted Poincar\'e inequalities. The proof is based on simple arguments of calculus of variations. 

Further, we  establish two general properties enjoyed by  the Stein discrepancy, holding whenever a Stein kernel exists:  Stein discrepancy is strictly decreasing along the CLT, and it  controls the skewness of a random vector. 
\end{abstract}

\section{Introduction}
What is known as Stein's method is a vast array of concepts and techniques for proving quantitative convergence of sequences of random variables to some limit. These ideas originated in the work of Stein \cite{Ste72, Ste86}, and have found many applications in the study of quantitative central limit theorems, Poisson and geometric approximation, concentration of measure, random matrix theory and free probability. We refer to the survey \cite{Ros11} for an overview of the topic. 

In this work, we shall be interested in one particular concept used in this setting: Stein kernels (also known as Stein factors) and their use in proving quantitative central limit theorems. To this end, let $\nu$ be a probability measure on $\R^d$. A matrix-valued function $\tau_{\nu} : \R^d \longrightarrow \mathcal{M}_d(\R)$ is said to be a \emph{Stein kernel} for $\nu$ (with respect to the standard Gaussian measure $\gamma$ on $\R^d$) if for any smooth test function $\varphi$ taking values in $\mathbb{R}^d$,  we have
\begin{equation} \label{eq_stein}
\int{x \cdot \varphi d\nu} = \int{\langle \tau_{\nu}, \nabla \varphi \rangle_{\mathrm{HS}} d\nu} 
\end{equation}
where $\langle \cdot, \cdot \rangle_{\mathrm{HS}}$ stands for the usual Hilbert-Schmidt scalar product on $\mathcal{M}_d(\R)$. For applications, it generally suffices to consider the restricted class of test functions $\varphi$ satisfying $\int (|\varphi|^2 + \|\nabla \varphi\|_{\mathrm{HS}}^2 )d\nu <\infty$, in which case both integrals in \eqref{eq_stein} are well-defined as soon as  $\tau_{\nu} \in L^2(\nu)$, provided $\nu$ has finite second moments.  We shall adopt this convention throughout.

In parts of the literature, the notion of Stein kernel is replaced by the relation
\begin{equation}\label{SKotherDef}
\int{x \cdot \nabla \varphi d\nu} = \int{\langle \tau_{\nu}, \Hess \varphi \rangle_{\mathrm{HS}} d\nu}
\end{equation}
for all smooth real-valued functions $\varphi$. This notion is slightly weaker compared to \eqref{eq_stein} since it only requires test functions that are gradients, but for some applications it still suffices. Our results will hold for either definition, but we shall adopt the stronger notion \eqref{eq_stein} throughout since the improvement comes for free.   

  The  motivation behind the definition is that, since the Gaussian measure is the only probability distribution satisfying the integration by parts formula
\begin{equation} \label{ibp_gauss}
\int{x \cdot \varphi d\gamma} = \int{\operatorname{div} (\varphi)d\gamma},
\end{equation} 
one can take the identity matrix, denoted by $\id$, as a Stein kernel if and only if the measure $\nu$ is  equal to $\gamma$. In this way,  the Stein kernel can be seen as a measure of how far $\nu$ is from being a standard Gaussian measure in terms of how much it violates the integration by parts formula \eqref{ibp_gauss}. Those kernels appear implicitly in many works on Stein's method, and have recently been the topic of more direct investigations \cite{AMV10, NPR10, NPS14a, NPS14b, LNP15}.

 The question of when a Stein kernel exists for a particular measure $\nu$ is a nontrivial one, and only a few results are known along this direction.     In dimension one, it suffices to have mean zero and a density with connected support to ensure existence.  Indeed, if $\nu$ has a density $p$ that does not vanish on the (possibly infinite) interval $(a,b)$, then the Stein kernel $\tau_{\nu}$   is unique up to sets of measure zero, and  is given by  
\begin{equation} \label{eq_1d}
\tau_{\nu}(x) := \frac{1}{p(x)}\int_x^{\infty}{yp(y)dy}.
\end{equation}
In general, however, Stein kernels are not necessarily unique  when they exist.   A   more detailed study of the one-dimensional case and its generalizations to non-Gaussian reference measures can be found in \cite{LRS17}. 

In higher dimension, existence of Stein kernels has been previously studied using the tools of Malliavin calculus \cite{NP12}. In particular, if a random variable can be realized as the image of a Gaussian random variable by a $C^{\infty}$ function with derivatives of at most polynomial growth, then a Stein kernel exists.  Another explicit formula for one-dimensional random variables that arise as smooth functions of some Gaussian vector was also obtained in \cite{Cha09}.    However, given a probability distribution, it may be difficult to find such a smooth function. For example, Brenier's theorem in optimal transport \cite{Bre91} tells us that under fairly general assumptions there exists a map sending a Gaussian random variable onto the distribution considered, but in general it is not smooth enough to apply the arguments of \cite{NP12}. 

Our main results are roughly divided into two categories:  sufficient conditions for existence of Stein kernels in arbitrary dimension, and general bounds on the so-called \emph{Stein discrepancy} which hold whenever a Stein kernel exists.   Specifically, we first show that if $\nu$ satisfies a Poincar\'e inequality, or more generally a converse weighted Poincar\'e inequality, then a Stein kernel exists.  This affirmatively  answers a question raised in \cite{NPS14b}.   In doing so, we  obtain bounds on the associated {Stein discrepancy} for measures satisfying a Poincar\'e inequality.  These estimates are dimension-free and  depend only on the second moment and the Poincar\'e constant.   We further establish two properties enjoyed by Stein discrepancy that hold in general, whenever a Stein kernel exists.  First, like entropy and Fisher information, Stein discrepancy is  monotone along the CLT.  Second, Stein discrepancy is bounded from below by the skewness of a random vector. 
 
 These results lead to optimal rates of convergence in the multidimensional central limit theorem in Wasserstein distance $W_2$, as well as entropic CLTs, with suboptimal rate. Our main estimate can also be reformulated as a quantitative improvement of the fact that among all isotropic measures, the standard Gaussian measure has the best Poincar\'e constant.  

%
%

\section{On existence of Stein kernels}

Let $\nu$ be a probability measure on $\R^d$. Henceforth, we  make the following assumption: 

\begin{assumption}
The measure $\nu$ is absolutely continuous with respect to the Lebesgue measure, and has finite second moment, i.e. $\int{|x|^2d\nu} < \infty$. 
\end{assumption}

We shall work in the Sobolev space $W^{1,2}_{\nu}$ of vector valued functions, which we define as as the closure of the set of all smooth vector-valued functions $f : \R^d \longrightarrow \R^d$ in $L^2(\nu)$, with respect to the usual Sobolev norm $\int{(|f|^2 + \|\nabla f\|_{\mathrm{HS}}^2)d\nu}$. We also define its restriction to the set of (vector-valued) functions with average zero 
$W^{1,2}_{\nu, 0} := W^{1,2}_{\nu}\bigcap \left\{f : \int{fd\nu} = 0\right\}.$

\begin{defn}
A function $\tau_{\nu} : \R^d \longrightarrow \mathcal{M}_d(\R)$ is a Stein kernel for $\nu$ if for any $\varphi \in W^{1,2}_{\nu}$ equation \eqref{eq_stein} holds. The Stein discrepancy is defined as 
$$S(\nu|\gamma)^2 := \inf \int{\|\tau_{\nu} - \mathrm{Id}\|_{\mathrm{HS}}^2d\nu},$$
where the infimum is taken over all Stein kernels of $\nu$, and takes value $+\infty$ if no Stein kernel exists. 
\end{defn}

 One of the main applications of Stein kernels is that bounds on the Stein discrepancy can be used to obtain rates of convergence in the central limit theorem, as discussed in Section \ref{sec:Appl}.

We now introduce the functional inequalities we shall use as criteria for existence of Stein kernels. 

\begin{defn}
A probability measure $\nu$ is said to satisfy a Poincar\'e inequality with constant $C_p$ if for any locally lipschitz function $f \in L^2(\nu)$ we have
$$\Var_{\nu}(f) \leq C_p \int{|\nabla f|^2d\nu}.$$
\end{defn}
A measure satisfying a Poincar\'e inequality is also said to have spectral gap.  The terminology comes from the fact that $C_p^{-1}$ is a lower bound on the smallest positive eigenvalue of the operator $-\Delta + \nabla H \cdot \nabla$ in $L^2(\nu)$, where $H = -\log \frac{d\nu}{dx}$. 

There is a vast literature on Poincar\'e inequalities, with many examples and abstract results giving sufficient conditions for one to hold. In particular, the class of measures satisfying a Poincar\'e inequality is stable under bounded perturbations and  tensor products, and  it contains the set of all log-concave probability measures. A more general sufficient condition for a measure with density $e^{-V}$ to have  spectral gap is 
$$\exists a \in (0,1), R \geq 0, c> 0 \text{ such that } a|\nabla V(x)|^2 - \Delta V(x) \geq c \hspace{3mm} \forall |x| \geq R,$$
which was obtained in \cite{BBCG08}. We refer to \cite{BGL14} for more background on Poincar\'e inequalities. 

We shall also consider a more general type of functional inequalities: 

\begin{defn}
A probability measure $\nu$ is said to satisfy a converse weighted Poincar\'e inequality with weight $\omega : \R \longrightarrow \R_+^*$ if for any locally lipschitz $f \in L^2(\nu)$ we have
$$\underset{c \in \R}{\inf} \hspace{1mm} \int{(f-c)^2\omega d\nu} \leq \int{|\nabla f|^2d\nu}.$$
\end{defn}

This definition originates from \cite{BL09}, and was further studied in \cite{CGGR10}. Such inequalities are related to measure concentration for heavy-tailed distributions. We could incorporate a constant in front of the Dirichlet form in the definition, but we have chosen to absorb it into the weight to reduce notations, so that a Poincar\'e inequality with constant $C_P$ corresponds to a converse weighted Poincar\'e inequality with constant weight $C_P^{-1}$. 

\subsection{Finite Poincar\'e constant ensures existence of a Stein kernel}

Our main result of this section is that a (converse weighted) Poincar\'e inequality ensures existence of a Stein kernel, and moreover that the Poincar\'e constant controls the Stein discrepancy. Stated more precisely, 

\begin{thm} \label{main_thm_stein}
Assume that $\nu$ is centered (i.e. has mean zero) and satisfies a converse weighted Poincar\'e inequality with weight $\omega$, and such that $\int{|x|^2\omega^{-1} d\nu} < \infty$. Then there exists a unique function $g \in W^{1,2}_{\nu, 0}$ such that $\tau_{\nu} = \nabla g$ is a Stein kernel for $\nu$. Moreover, 
\begin{equation} \label{steinDiscBound}
\int{\|\tau_{\nu}\|_{\mathrm{HS}}^2d\nu} \leq \int{|x|^2\omega^{-1} d\nu}.
\end{equation}

In particular, if $\nu$ is centered and satisfies a Poincar\'e inequality with constant $C_p$, the above result applies and 
\begin{align} \label{steinDiscBound_poinca}
\int{\|\tau_{\nu}\|_{\mathrm{HS}}^2d\nu} \leq C_p \int{|x|^2d\nu}
\end{align}
so that the Stein discrepancy satisfies
$$S(\nu|\gamma)^2 \leq (C_p-2) \int{|x|^2d\nu}  + d$$
%
\end{thm}

The centering assumption on $\nu$ is necessary for the theorem to hold.  Indeed, a necessary condition for existence of a Stein kernel is that $\nu$ is centered, seen by taking $\varphi = 1$ in the defining equation \eqref{eq_stein}.     

In most situations, we shall be using the above bounds for measures satisfying a Poincar\'e inequality and with second moment normalized with respect to dimension (e.g., as is the case for isotropic measures): 
\begin{cor} \label{cor_main_thm}
Let $\nu$ be a centered probability measure on $\R^d$ satisfying a Poincar\'e inequality with constant $C_p$, normalized so that $\int |x|^2 d\nu=d$. Then 
$$S(\nu|\gamma)^2 \leq d(C_p-1).$$
\end{cor}
\noindent A few remarks are in order:  
\begin{itemize}
\item The standard Gaussian measure $\gamma$ has Poincar\'e constant $C_p=1$, so the above estimates dictate $S(\gamma|\gamma)=0$, as desired.  
\item Stein discrepancy  is additive on product measures (i.e., $S(\nu^{\otimes k}|\gamma^{\otimes k})^2 = k S(\nu|\gamma)^2$), whereas the Poincar\'e constant is independent of dimension (i.e., $C_p(\nu^{\otimes k}) = C_p(\nu)$).  Thus, our estimates are dimension-free in nature.
\item A converse weighted Poincar\'e inequality is by no means  necessary  for existence of a Stein kernel. In dimension one, the formula \eqref{eq_1d} works in more general situations.  We will see another multi-dimensional example further on.
\end{itemize}

\begin{proof}[Proof of Theorem \ref{main_thm_stein}]
The result follows from an application of the Lax-Milgram theorem \cite{LM54}.  Indeed, $\int \langle \nabla f, \nabla h\rangle_{\mathrm{HS}}d \nu$ is a continuous bi-linear functional on $W^{1,2}_{\nu, 0}\times W^{1,2}_{\nu, 0}$, and dominates the weighted Sobolev norm $\int{\omega |f|^2 + |\nabla f|^2d\nu}$ for non-constant functions by the assumption that $\nu$ satisfies a converse weighted Poincar\'e inequality. Finally, $ f \longrightarrow \int{f \cdot x d\nu}$ on $W^{1,2}_{\nu, 0}$ is a continuous linear form since for any $\vec{c} = (c_1,.., c_d) \in \R^d$ we have
\begin{align*}
\int{x \cdot f d\nu} &= \int{x \cdot (f - \vec{c})d\nu} \\
&\leq \left(\int{\omega^{-1}|x|^2d\nu}\right)^{1/2}\left(\underset{\vec{c}}{\inf} \hspace{1mm} \int{\omega \sum |f_i - c_i|^2d\nu}\right)^{1/2} \\
&\leq \left(\int{\omega^{-1}|x|^2d\nu}\right)^{1/2}\left(\int{|\nabla f|^2d\nu}\right)^{1/2}.
\end{align*}
Hence there exists a unique $g \in W^{1,2}_{\nu, 0}$ such that 
\begin{equation} \label{eq_stein_proof}
\int{\langle \nabla g, \nabla f\rangle_{\mathrm{HS}} d\nu} = \int{ x\cdot f d\nu}
\end{equation}
for any $f \in W^{1,2}_{\nu, 0}$, and in particular $\nabla g$ is a Stein kernel. 

In addition,  $g$ minimizes the functional $J(f) := \frac{1}{2}\int \|\nabla f\|_{\mathrm{HS}}^2d \nu
 - \int{x \cdot f d\nu}$. Indeed, 
\begin{align*}
J(f) &=  \frac{1}{2}\int  \|\nabla f\|_{\mathrm{HS}}^2 d \nu - \int{x \cdot f d\nu}\\
&=  \frac{1}{2}\int \|\nabla f\|_{\mathrm{HS}}^2 d \nu - \int{\langle \nabla g, \nabla f\rangle_{\mathrm{HS}} d\nu} \\
&\geq -\frac{1}{2}\int \|\nabla g\|_{\mathrm{HS}}^2 d \nu = J(g)
\end{align*}
where we have just applied \eqref{eq_stein_proof} to integrate by parts to go from the first to the second line, and applied the Cauchy-Schwarz inequality for the final lower bound, while applying again \eqref{eq_stein_proof} with $f = g$ yields $J(g) =  -\frac{1}{2}\int \|\nabla g\|_{\mathrm{HS}}^2 d \nu$. The Cauchy-Schwarz inequality and the converse weighted Poincar\'e inequality for $\nu$ then give, after a simple computation,    
\begin{align}
 -\frac{1}{2}\int \|\nabla g\|_{\mathrm{HS}}^2 d \nu = J(g) &\geq \frac{1}{2} \int |\nabla g|^2d\nu -  \left( \int \omega|g|^2 d\nu \right)^{1/2}\left( \int |x|^2\omega^{-1} d\nu \right)^{1/2}\\
&\geq  -\frac{1}{2}\int \omega^{-1}|x|^2 d\nu, 
\end{align}
establishing \eqref{steinDiscBound}. 
%
%
\end{proof}

\begin{rque} \label{rmk_no_sg}  \normalfont
Even when $\nu$ does not satisfy a (converse weighted) Poincar\'e inequality, if $g\in W^{1,2}_{\nu, 0}$ minimizes the functional $J: f \mapsto \frac{1}{2}\int \|\nabla f\|_{\mathrm{HS}}^2 - \int{x \cdot f d\nu}$, then $\nabla g$ is a Stein Kernel for $\nu$.  To see this, consider a perturbation in the direction $h \in W^{1,2}_{\nu, 0}$, which gives:
\begin{align*}
0 &\leq J(g + \varepsilon h)-J(g) \\
&=  {\varepsilon}\left(\int \langle \nabla g,\nabla h\rangle  d \nu - \int x \cdot h d \nu \right)+ \frac{\varepsilon^2}{2}\int |\nabla h|^2 d \nu.
\end{align*}
Letting $\varepsilon\downarrow 0$ shows that  $\int \langle \nabla g,\nabla h\rangle  d \nu \geq  \int x \cdot h d \nu$.  Replacing $h$ by $-h$ gives the reverse inequality.   

Hence, a sufficient condition for existence of a Stein kernel is that the functional $J$ has a minimum.   Stated another way, there exists a finite constant $c>0$ such that 
\begin{align} \label{generalized_condition}
\left(\int{x \cdot f d\nu}\right)^2 \leq c \int \|\nabla f\|_{\mathrm{HS}}^2 d \nu   ~~~\forall f \in W^{1,2}_{\nu, 0},
\end{align}
and moreover, equality is attained for some nonzero function $g$. This should be compared against the definition of the Poincar\'e inequality. 
\end{rque}

Converse weighted Poincar\'e inequalities have been established for a large class of heavy-tailed probability distributions via Lyapunov function techniques in \cite{CGGR10}. Here are some examples from \cite{BL09, CGGR10}: 

\begin{cor}
Stein kernels exist for the following probability distributions on $\R^d$: 

(i) Generalized Cauchy distributions $\nu_{\beta}(dx) := Z^{-1}(1 + |x|^2)^{-\beta}$ for $\beta > \max((d + 4)/2, d)$;

(ii) Probability measures of the form $\nu(dx) = Z^{-1}e^{-V(x)^p}$ with $V$ convex and $p > 0$, as soon as $\int{|x|^{2 + 2(1-p)}d\nu} < \infty$. In particular, subexponential distributions with density proportional to $e^{-|x|^p}$ with $p \in (0,1)$. 
\end{cor}

Note that these examples typically do not satisfy a classical Poincar\'e inequality. 

Of course, there exist probability measures that satisfy the \eqref{generalized_condition} condition without satisfying a (converse weighted) Poincar\'e inequality. For example, if we consider two disjoint closed annuli $C_1$ and $C_2$ that are centered around the origin, then the uniform probability measure on $C_1 \cup C_2$ does not satisfy a converse weighted Poincar\'e inequality, yet it does satisfy \eqref{generalized_condition}.

We conclude this section by noting that, as pointed out in \cite{LNP15}, for log-concave probability measures (which always satisfy a Poincar\'e inequality \cite{BBCG08}) there is a reverse version of our inequality: 

\begin{prop} \label{prop_rev_stein_est}
Let $\nu$ be a centered log-concave measure. Then for some numerical constant $C$, 
$$C_p \leq C(1 +  S(\nu|\gamma)^2). $$
\end{prop}

This statement, combined with our main result, tells us that for log-concave measures, controlling the Stein discrepancy and controlling the Poincar\'e constant are equivalent. At first glance, the above estimate does not capture the dimension-free nature of the Poincar\'e constant. This may be a necessary downside of such  bounds, since if we consider a measure of the form $\nu = \gamma_{d-1} \otimes \tilde{\nu}$ it behaves in the correct way, as the Poincar\'e constant is at least as bad as that of the projection along the worst direction. 

As mentioned in \cite{LNP15}, Proposition \ref{prop_rev_stein_est} is obtained by combining the moment bound of Theorem 2.8 in \cite{LNP15} and Milman's results on obtaining estimates on Poincar\'e constants of log-concave measures by the worst variance of $1$-lipschitz functions \cite{Mil09}.

\subsection{Extension to non-Gaussian reference measures}

Theorem \ref{main_thm_stein} also generalizes to Stein kernels with non-Gaussian reference measures. Such an extension is natural in the framework of the generator approach to Stein's method, where an integration by parts formula for a given measure is obtained by finding a Markov generator that leaves the considered measure invariant. This approach was pioneered in \cite{Bar90, Got91}. Stein's method for the approximation of non-Gaussian reference measures has had some successful applications in the study of convergence of Markov Chain Monte Carlo algorithms \cite{Dia04} and for generalizations of the fourth moment theorem \cite{ACP13}. The Gaussian functional inequalities of \cite{LNP15} were also extended to a class of non-Gaussian measures, using arguments from Bakry-\'Emery calculus.  

We can extend Theorem \ref{main_thm_stein} to the situation where the Gaussian measure is replaced by a general reference measure $\mu = e^{-V}dx$, where $V: \mathbb{R}^d \to \mathbb{R}$ is a smooth function. In this situation, a Stein kernel of a measure $\nu$ with respect to $\mu$ is defined by the relation
\begin{equation}
\int{\nabla V \cdot \varphi d\nu} = \int{\langle \tau_{\nu}, \nabla \varphi \rangle_{\mathrm{HS}}d\nu} \hspace{3mm} \forall \varphi \in W^{1,2}_{\nu}.
\end{equation}

Applying the same arguments as for the Gaussian case, we obtain
\begin{thm} Let $\mu = e^{-V}dx$, where $V: \mathbb{R}^d \to \mathbb{R}$ is  smooth.
Assume that $\nu$ satisfies a Poincar\'e inequality with constant $C_p$, that $\int{\nabla V d\nu} = 0$ and that $\int{|\nabla V|^2d\nu} < \infty$. Then there exists a Stein kernel for $\nu$, relative to $\mu$, of the form $\tau_{\nu} = \nabla g$ for some $g \in W^{1,2}_{\nu}$. Moreover, we have the bound
\begin{align}
\int \|\tau_{\nu}\|^2_{\mathrm{HS}}d\nu \leq C_p \int |\nabla V|^2 d\nu.
\end{align}  
\end{thm}

Note that for polynomial potentials $V$, the finiteness of $\int{|\nabla V|^2d\nu}$ automatically follows from the Poincar\'e inequality. As in the previous section, this result can easily be generalized to cover measures satisfying a converse weighted Poincar\'e inequality. 

\section{General  bounds on the Stein discrepancy}
\subsection{Stein discrepancy controls skewness}
Above, it was shown that in presence of a suitable Poincar\'e inequality, the Stein discrepancy is controlled from above by second moments.   Here, we  establish a complementary lower bound on the Stein discrepancy in terms of skewness:

\begin{thm} \label{thm_stein_lb}
Let $X = (X_1, X_2, \dots, X_d)$ have law $\nu$.  If $\nu$ is  isotropic  with finite fourth moment, then
$$S(\nu|\gamma)^2 \geq \frac{1}{9}\underset{i=1}{\stackrel{d}{\sum}}  |\mathbb{E}[X_i^3]|^2.$$
\end{thm}

\begin{proof}   
First, we shall reduce the problem to the one-dimensional case. Let $\tau$ be a Stein factor for $\nu$. Then $\tau^i(x) := \mathbb{E}[\tau_{ii}(X)|X_i = x]$ is a Stein kernel for $X_i$. Moreover, we have 
\begin{align*}
S(\nu|\gamma)^2 &= \int{ \sum_{1\leq i,j\leq d}  |\tau_{i,j} - \delta_{i,j}|^2d\nu} \\
&\geq \int{ {\sum_{i=1}^d}|\tau_{i,i} - 1|^2d\nu} \geq \sum_{i=1}^d \mathbb{E}[|\tau^i(X_i) - 1|^2] \geq \sum_{i=1}^d S(\nu_i | \gamma_i)^2,
\end{align*}
where $\nu_i$ is the law of $X_i$ and $\gamma_i$ is the $i$th marginal of $\gamma$.
Hence it is enough to prove the theorem when $d=1$. 

Now let $X$ be a real-valued random variable with mean zero and unit variance, and let $\nu_n$ be the law of the standardized sum $\frac{1}{\sqrt{n}}\underset{i=1}{\stackrel{n}{\sum}} \hspace{1mm} X_i$, where the $X_i$'s are independent copies of $X$. Then 
\begin{align}
S(\nu|\gamma)^2 \geq nS(\nu_n|\gamma)^2 \label{CLTdecay}
\end{align}
 for any $n \geq 1$ (see for example Section 2.5 in \cite{LNP15} and Theorem \ref{thm:SKconvolution} in the next section). Moreover, it was established in \cite{LNP15} that the Stein discrepancy is always larger than the Wasserstein distance $W_2$ to the standard Gaussian measure. Hence $S(\nu_n|\gamma)^2 \geq W_2(\nu_n, \gamma)^2$ for all $n\geq1$. Finally, Rio established in \cite{Rio11} that under our assumptions, $\sqrt{n} W_2(\nu_n, \gamma) \longrightarrow \frac{1}{3}|\mathbb{E}[X^3]|$, which concludes the proof. 
\end{proof}

\subsection{Strict Monotonicity of the Stein Discrepancy in the CLT}


Monotonicity of information measures along the CLT have a long history, going back to Shannon's conjecture on the monotonicity of the entropy, which was eventually resolved in \cite{ABBN04b}. More specifically, if $S_n = \tfrac{1}{\sqrt{n}}\sum_{i=1}^n X_i$, where $X_1, \dots,X_n$ are i.i.d.~isotropic random vectors, then both the  entropy and Fisher information of $S_n$ with  respect to the standard Gaussian measure are non-increasing in $n$.  Following Artstein, Ball, Barthe and Naor's proof of this fact, several generalizations and alternative proofs have been discovered \cite{tulino2006monotonic, Shlyakhtenko, madiman2006, madiman2007generalized, madiman2009, Cou16b}.  

Since Stein discrepancy relates to both Fisher information and entropy in various ways \cite{LNP15}, it is natural to conjecture that it also is non-increasing along the CLT.  It turns out that this is indeed the case and, in fact, it is strictly decreasing.  The following generalizes \eqref{CLTdecay} along these lines:

\begin{thm}\label{thm:SKconvolution}
Let $\nu$ be an isotropic probability measure on $\mathbb{R}^d$, and let $\nu_n$ denote the law of $S_n = \tfrac{1}{\sqrt{n}}\sum_{i=1}^n X_i$, where $X_1, \dots,X_n$ are i.i.d.~with law $\nu$.  Then
 $$S(\nu_n|\gamma)^2 \leq  {\frac{m}{n}}S(\nu_m|\gamma)^2 ~~~~~~~~1\leq m\leq n.$$
\end{thm}
 \begin{proof}
For $m\geq 1$, let $\tau_m$ denote a Stein kernel associated with $S_m$. 
 We may assume that such a $\tau_m$ exists, since if it does not, the claim is vacuous. We shall first show that for all $n\geq m$, the function
\begin{align}
\tau_n(s_n) =  \E[\tau_m(S_m) | S_n=s_n]\label{TmTnIdentity}
\end{align}
is a valid Stein kernel for $S_n$. Indeed, for  $s_n = \sqrt{\tfrac{m}{n}}s_m +  \tilde{s}$, any smooth function $\varphi$ evaluated on $s_n$ may also be considered as a smooth function of $s_m$ for each fixed $\tilde{s}$.   In particular, the chain rule directly yields
\begin{align}
\nabla_{s_m} \varphi(s_n) = \sqrt{\frac{m}{n}} \nabla_{s_n}\varphi(s_n).\notag
\end{align}
Therefore, starting with linearity of expectation and defining   $\tilde{S}:= S_n -\sqrt{\tfrac{m}{n}}S_m$, we may write
\begin{align*}
\E \langle S_n ,   \varphi(S_n) \rangle   &= \sqrt{\frac{n}{m}} \E \langle S_m ,   \varphi(S_n) \rangle\\
&=  \sqrt{\frac{n}{m}}  \E[  \E [\langle S_m ,  \varphi(S_n) \rangle | \tilde{S} ] ]\\
&=  \sqrt{\frac{n}{m}}  \E[  \E [\langle \tau_m(S_m) ,  \nabla_{s_m} \varphi(S_n) \rangle_{\mathrm{HS}} | \tilde{S} ] ]\\
&=     \E [ \langle \tau_m(S_m) , \nabla_{s_n} \varphi(S_n) \rangle_{\mathrm{HS}}    ]\\
&=       \E  \langle \E[\tau_m(S_m)|S_n] , \nabla \varphi(S_n) \rangle_{\mathrm{HS}}  ,
\end{align*}
establishing \eqref{TmTnIdentity} is a valid Stein kernel.

Following \cite{Cou16b},  if a function $\vartheta:\mathbb{R}^d\to \mathbb{R}$ satisfies $\E \vartheta(S_m)=0$, then  
\begin{align}
& \E [ \left|\E[\vartheta(S_m)|S_n] \right|^2  ]   \leq  \frac{m}{n}  \E[ \left| \vartheta(S_m)\right|^2  ]   &1\leq m\leq n. \label{maxCorr}
\end{align}
This inequality is due to Dembo, Kagan and Shepp \cite{dembo2001remarks}; see also Kamath and Nair \cite{kamath2015strong}.   Now, $\E[\tau_m(S_m)] = \id$, so a direct application of \eqref{maxCorr} to the coordinates of $\tau_m - \id$ yields
$$\E\| \tau_n(S_n) - \id \|_{\mathrm{HS}}^2 = \E[ \| \E[\tau_m(S_m) - \id | S_n] \|_{\mathrm{HS}}^2 ] \leq \frac{m}{n} \E\| \tau_m(S_m) - \id \|_{\mathrm{HS}}^2.$$
Taking the infimum over all valid Stein kernels $\tau_m, \tau_n$ finishes the proof.

\end{proof}

\begin{rque}
The same result holds if \eqref{SKotherDef} is adopted as the definition of a Stein kernel.
\end{rque}

\section{Applications}\label{sec:Appl}

\subsection{Quantitative central limit theorems}

We shall now discuss some applications of the bounds to quantitative central limit theorems in Wasserstein distance $W_2$, which is defined as
$$W_2(\mu, \nu)^2 := \underset{\pi}{\inf} \int{|x-y|^2d\pi(x,y)},$$
where the infimum is taken over all couplings $\pi$ of the probability measures $\mu$ and $\nu$. We refer the reader to the textbook \cite{Vill03} for more information about Wasserstein distances and optimal transport. 

\begin{thm} \label{thm_clt_quant_w2}
Let $X_1,..X_n$ be independent centered, isotropic random variables. Assume that the law of $X_i$ satisfies a Poincar\'e inequality with constant $C_i$, and let $\nu_n$ be the law of $\frac{1}{\sqrt{n}}\underset{i=1}{\stackrel{n}{\sum}} \hspace{1mm} X_i$. Then 
$$W_2(\nu_n, \gamma)^2 \leq \frac{d}{n^2}\underset{i= 1}{\stackrel{n}{\sum}} (C_i - 1).$$
In particular, if the $X_i$ are i.i.d., and their law $\nu$ is centered, isotropic and satisfies a Poincar\'e inequality with constant $C_p$, then 
$$W_2(\nu_n, \gamma)^2 \leq {\frac{d(C_p-1)}{n}}.$$
\end{thm}

\noindent We remark that the rate $W_2(\nu_n, \gamma) = O(1/\sqrt{n})$ in the CLT for i.i.d. random variables is known to be optimal in general. Moreover, the dependence on the dimension is sharp, since it cannot be improved for product measures. To our knowledge, this seems to be the first result with sharp dependence on both the dimension and on $n$ for $W_2$ and with assumptions satisfied by a large class of probability measures. A similar result can be obtained with converse weighted Poincar\'e inequalities, with the same sharp rate but a less explicit prefactor. 

In the i.i.d. case, there are several similar results already present in the literature. In dimension one, a  more general result has been obtained by Rio in \cite{Rio09, Rio11}, where a finite fourth moment suffices. He also obtained convergence in stronger transport distances when the random variable has a finite exponential moment, which is a weaker assumption than our use of a Poincar\'e inequality. The proofs rely on an explicit representation of transport maps involving the repartition function of $\nu$, which is unavailable in higher dimensions. Subsequently,  Bobkov \cite{Bob13}  combined optimal rates in the entropic CLT \cite{BCG14} with Talagrand's inequality to conclude $O(1/\sqrt{n})$ convergence of $W_2(\nu_n, \gamma)$ in dimension one, but left open the problem in higher dimensions.

In the multidimensional setting, Zhai \cite{Zha16} has recently established  that for  random variables in dimension $d$, we have $W_2(\nu_n, \gamma) \leq \frac{5\sqrt{d}\beta(1 + \log n)}{\sqrt{n}}$, under the boundedness assumption  $|X|\leq \beta$ almost surely. His assumptions are not directly comparable with ours, since bounded random variables do not necessarily satisfy a Poincar\'e inequality, while there are many examples of unbounded random variables that do satisfy one.  However, it is true that every bounded random variable regularized via convolution with a Gaussian measure of arbitrarily small variance does satisfy a Poincar\'e inequality \cite{Bardet16},  which suggests that Zhai's result may potentially be improved to have optimal dependence on both dimension and $n$.  Unfortunately, the bounds on the Poincar\'e constant obtained in \cite{Bardet16} are exponential in $\beta$, so it is not clear whether Zhai's result may  be recovered from our own via this route.  In situations where both estimates apply, the bound in the present work will typically be smaller. For example, for high-dimensional product measures, the Poincar\'e constant is independent of the dimension, while $\beta$ would be of order of $\sqrt{d}$. Moreover, we eliminate the extra $\log n$ factor. It may be relevant to point out that both our assumptions and those of \cite{Zha16} fit in the framework of random variables with a finite exponential moment. 

Also in  higher dimensions, Bonis showed in \cite{Bon16} that, under the moment constraint $\mathbb{E} \|X\|^{2+m}<\infty$ for $m\in [0,2]$,  we have the asymptotic rate  $W_2(\nu_n, \gamma) = O(n^{-1/2 + (2-m)/4})$.   However, in the case $m=2$, the  prefactor (which does not appear explicitly in Bonis' work) seems to have a suboptimal dependence on the dimension $d$ \cite{Bon16p}. 

In dimension 1, and for log-concave measures in higher dimension, the works \cite{BN12, BBN03, JB08} can be used to obtain a sharp rate of convergence in relative entropy when a Poincar\'e inequality holds, which implies convergence in $W_2$. These results however would rely on a bound on the relative entropy of $\nu$, which in general would give a worse prefactor in the bound. 


\begin{proof}[Proof of Theorem \ref{thm_clt_quant_w2}]
The proof hinges on the fact that 
$$\bar{\tau}_n(x) := \mathbb{E}\left[\frac{1}{n}\underset{i=1}{\stackrel{n}{\sum}} \hspace{1mm} \tau_i(X_i) \left.\right| \frac{1}{\sqrt{n}}\underset{i=1}{\stackrel{n}{\sum}} \hspace{1mm} X_i = x \right]$$
is a Stein kernel for the law of $\frac{1}{\sqrt{n}}\underset{i=1}{\stackrel{n}{\sum}} \hspace{1mm} X_i$, where $\tau_i$ is a Stein kernel for the law of $X_i$. In the i.i.d. case, we already proved this fact in the proof of Theorem \ref{thm:SKconvolution}, and the proof in the non-identicaly distributed case is exactly the same. As a consequence, using the fact that conditional expectation is an $L^2$-projection,  
\begin{align*}
W_2(\nu_n, \gamma)^2 &\leq S(\nu_n)^2 \leq \int{|\bar{\tau}_n - \operatorname{Id}|^2d\nu_n} \\
&\leq \frac{1}{n^2}\underset{i, j = 1}{\stackrel{n}{\sum}} \hspace{1mm} \int{\langle \tau_i(x_i) - \operatorname{Id}, \tau_j(x_j) - \operatorname{Id} \rangle d\nu_i(x_i)d\nu_j(x_j)} \\
&= \frac{1}{n^2}\underset{i= 1}{\stackrel{n}{\sum}} \hspace{1mm} \int{|\tau_i(x_i) - \operatorname{Id}|^2d\nu_i(x_i)} \\
&\leq \frac{d}{n^2}\underset{i= 1}{\stackrel{n}{\sum}} (C_i - 1)
\end{align*}
where the last bound is obtained by applying Theorem \ref{main_thm_stein}. This concludes the proof. 
\end{proof}

It is also possible to extend the method to \emph{weakly} dependent random variables, using a standard splitting trick: 

\begin{thm}
Let $(X_i)$ be a sequence of centered isotropic random variables, and assume that there exists a $k$ such that as soon as $|i-j| > k$ then $X_i$ and $X_j$ are independent. Assume moreover that $\mathbb{E}[X_i \cdot X_j] = 0$ for all distinct $i,j$, and that the law of each random variable satisfies a Poincar\'e inequality with uniform constant $C_p$. Then
$$W_2(\nu_n, \gamma)^2 \leq \frac{d(C_p-1) + 2d(k-1)}{\lfloor n/k \rfloor}$$
where $\nu_n$ is the law of $S_n := \frac{1}{\sqrt{n}}\underset{i = 1}{\stackrel{n}{\sum}} \hspace{1mm} X_{i}$. 
\end{thm}

As an example, this theorem applies to $X_i := \underset{j=0}{\stackrel{k-1}{\prod}} \hspace{1mm} f_j(Y_{i + j})$ with the $Y_i$ i.i.d. random variables satisfying some Poincar\'e inequality and the $f_j$ bounded lipschitz functions with mean zero and unit variance. 

\begin{proof}
We can define the partial sums
$$S^j_m := \frac{1}{\sqrt{m}}\underset{i = 0}{\stackrel{m-1}{\sum}} \hspace{1mm} X_{i(k + 1) + j}$$
for $j = 1,..,k$. Then each $S^j_m$ is a sum of independent random variables, so we can apply Theorem \ref{thm_clt_quant_w2} to obtain convergence in $W_2$ distance to the Gaussian, i.e. 
$$W_2(\nu_{j,m}, \gamma)^2 \leq \frac{(C_p -1)d}{m}$$
where $\nu_{j,m}$ is the law of $S^j_m$. Since $S_{km} = \frac{1}{\sqrt{k}}\underset{j=1}{\stackrel{k}{\sum}} \hspace{1mm} S^j_m$, it is easy to check that $W_2(\nu_{km}, \gamma)^2 \leq \frac{1}{k}\sum W_2(\nu_{j,m}, \gamma)^2 \leq \frac{(C_p -1)d}{m}$. Moreover for any $j \in \{1,..,k-1\}$, $$W_2(\nu_{km + j}, \gamma)^2 \leq \frac{km}{km+j}W_2(\nu_{km}, \gamma)^2 + \frac{2dj}{km+j} \leq \frac{(C_p -1)d}{m} + \frac{2d(k-1)}{km},$$ where we split the sum and used the fact that the Wasserstein distance is bounded by the second moment to control the contribution of $X_{mk+1},..,X_{mk+j}$. 
\end{proof}

\subsection{Entropy bounds}

Let $H_{\gamma}(\nu) := \int d\nu \log \tfrac{d \nu}{d\gamma}$ denote the entropy of $\nu$ relative to $\gamma$ and $I_{\gamma}(\nu) :=\int{|\nabla \log f|^2d\gamma}$ denote the relative Fisher information.  The HSI inequality of \cite{LNP15} states that
$$H_{\gamma}(\nu) \leq \frac{S(\nu)^2}{2}\log\left(1 + \frac{I_{\gamma}(\nu)}{S(\nu)^2}\right).$$
As a consequence, we also have the following rate of convergence in the entropic CLT: 

\begin{prop}
Assume that $\nu$ satisfies a Poincar\'e inequality with constant $C_p$, and satisfies the Fisher information bound $I_{\gamma}(\nu) \leq \alpha d$. Then we have
$$H_{\gamma}(\nu_n) \leq \frac{d(C_p-1)}{2n}\log\left(1 + \frac{\alpha n}{(C_p - 1)}\right)$$
\end{prop}

Convergence to the Gaussian measure in entropy is strictly stronger than convergence in $W_2$, due to Talagrand's inequality \cite{Tal96}. The choice of scaling in the dimension for the upper bound on the Fisher information reflects the fact that for product measures, it is of order $d$. In dimension one, the actual rate of convergence in the entropic CLT is asymptotically    $1/n$ under a fourth moment condition \cite{BCG13}, and non-asymptotically $1/n$ if the entropy of a single random variable is bounded \cite{BCG14} (with a prefactor that is exponential in the entropy). When the Poincar\'e inequality holds a non-asymptotic rate was obtained in \cite{BBN03, ABBN04}, and extended to multi-dimensional random vectors having log-concave density in \cite{BN12}. Related results in dimension one were obtained in \cite{JB08}.

\subsection{Fisher information bounds}

In this section, we shall combine our main estimate with results of \cite{NPS14b} to obtain bounds on the Fisher information of a sum of independent random variables, to which we add a small Gaussian noise. To this end, recall  that the Fisher information of $\nu$ relative to $\gamma$ is defined as $I_{\gamma}(\nu) := \int{|\nabla \log f|^2d\gamma}$, where $f = \frac{d\nu}{d\gamma}$. After applying Theorem V.3 in \cite{NPS14b}, we get

\begin{prop}
Let $(X_i)$ be a collection of independent centered isotropic random variables in $\R^d$ with Poincar\'e constants $C_i$.  Let  $W_n :=\sqrt{1-t}Z +  \frac{\sqrt{t}}{\sqrt{n}}\underset{i=1}{\stackrel{n}{\sum}} \hspace{1mm} X_i,$ where $Z$ is a standard Gaussian random variable independent of the $X_i$. If $\nu_n^t$ denotes the law  of $W_n$, then 
$$I_{\gamma}(\nu_n^t) \leq \frac{t^2}{n^2(1-t)}\underset{i=1}{\stackrel{n}{\sum}} \hspace{1mm} (C_i-1)d.$$
In particular, if the $X_i$ satisfy a Poincar\'e inequality with same constant $C_p$, then
$$I_{\gamma}(\nu_n^t) \leq \frac{t^2(C_p -1)d}{n(1-t)}.$$
\end{prop}

Due to Cramer's law \cite{Cr36}, weak convergence of $\nu_n^t$ to $\gamma$ is equivalent to convergence of $\frac{1}{\sqrt{n}}\underset{i=1}{\stackrel{n}{\sum}} \hspace{1mm} X_i$. Unfortunately, Cramer's law is unstable in general \cite{BCG12}, so we cannot directly deduce quantitative closeness of $X$ to a Gaussian if $X + Z$ is close to Gaussian for a general random variable $X$ (although the counterexamples of \cite{BCG12} do not seem to satisfy a Poincar\'e inequality, so it may be that under such an extra assumption Cramer's law would be stable). 

Rates of convergence in Fisher information of order $O(n^{-1})$ in dimension one when the information of a single variable is finite and under a Poincar\'e inequality have been obtained in \cite{JB08}.  In higher dimension, a quantitative bound on the difference between Fisher informations of $\nu_1$ and $\nu_2$ was obtained in \cite{johnson2004information}, but does not readily lead to a quantitative central limit theorem. 

\begin{rque}
Instead of using the results of \cite{NPS14b}, it is possible to derive upper bounds on $I_{\gamma}(\nu_n^t)$ by $W_2(\nu_n, \gamma)^2$ using the gradient flow structure of the Ornstein-Uhlenbeck flow, as done for example in Theorem 24.16 of \cite{Vil08}, and apply our bounds on the rate of convergence in $W_2$ distance to conclude. 
\end{rque}

\subsection{Stability of the Poincar\'e constant under a second moment constraint}

Combined with the previously mentioned fact that Stein discrepancy controls $W_2$ distance to $\gamma$, Corollary \ref{cor_main_thm} implies the following estimate: 

\begin{thm}\label{PIstable}
Let $\nu$ be a centered  probability measure on $\R^d$, normalized so that $\int |x|^2 d\nu = d$. Then its Poincar\'e constant $C_p$ satisfies
$$C_p \geq 1 + \frac{W_2(\nu, \gamma)^2}{d}.$$
\end{thm}

This estimate is a quantitative reinforcement of the fact that among all  probability measures with the same second moment, the Gaussian has the best Poincar\'e constant. More generally, it is a reinforcement of the fact that, given a sequence of centered measures $(\nu_n)$ with $\int{|x|^2d\nu_n} = d$, if their Poincar\'e constants converge to $1$, then $\nu_n$ weakly converges to the standard Gaussian \cite{BU84}. Once again, we note that this estimate depends optimally in the dimension. 

In a different direction, De Philippis and Figalli \cite{DF16} recently showed a similar quantitative stability result among a different class of measures: for densities that are of the form $e^{-V}\gamma$ with $V$ convex, the Gaussian has the worst Poincar\'e constant, and we have a deficit of the form $1 - C_p \leq \epsilon \Longrightarrow W_1(\nu, \gamma) \leq C(d, \alpha)|\log \epsilon |^{-1/4 + \alpha}$ for any $\alpha > 0$, for $\epsilon$ small enough. Our results are not directly comparable, since they concern completely different classes of measures. We just note that the dependence in $\epsilon$ in the result of \cite{DF16} is not expected to be sharp.  Indeed, in dimension one they show that $W_1(\nu, \gamma) \leq C\epsilon$. In spirit, this question is also similar to the stability problem for the Szeg\"o-Weinberger inequality, that was solved in \cite{BP11}. 

Finally, we observe that Theorem \ref{main_thm_stein} leads to the more general analogous  result for measures satisfying a converse weighted Poincar\'e inequality.  Although it is easily seen that such inequalities are stable under log-bounded transformations of the measure \cite{CGGR10}, the following appears to be the first quantitative stability result along these lines:
\begin{thm}
Let $\nu$ be a centered  probability measure on $\R^d$, normalized so that $\int |x|^2 d\nu = d$.  If $\nu$ satisfies a converse weighted Poincar\'e inequality with weight function $\omega$, then
$$\frac{1}{d}\int |x|^2 \omega^{-1} d\nu  \geq 1 + \frac{W_2(\nu, \gamma)^2}{d}.$$
\end{thm}
By  H\"{o}lder's inequality, the following corollary is immediate:
\begin{cor}
Let $\nu$ be a centered  probability measure on $\R^d$, normalized so that $\int |x|^2 d\nu = d$.   If $\nu$ satisfies a converse weighted Poincar\'e inequality with weight function $\omega$, then for H\"older conjugate exponents $1\leq p,q\leq \infty$, 
$$   {\| \tfrac{1}{d}|x|^2 \|_{L^p(\nu)}} \| \omega^{-1} \|_{L^q(\nu)}  \geq 1 + \frac{W_2(\nu, \gamma)^2}{d}.$$
\end{cor}
\noindent Of course, Theorem \ref{PIstable} coincides with the special case where $p=1$ and $q=\infty$.

\vspace{5mm}

\textbf{Acknowledgments} This work benefited from support from the France-Berkeley fund and the Labex CIMI. M. F. was partly supported by NSF FRG grant DMS-1361122 and Project EFI (ANR-17-CE40-0030) of the French National Research Agency (ANR). T.C. and A.P. were supported in part by NSF Grants CCF-1528132 and CCF-0939370 (Center for Science of Information). We thank Thomas Bonis, Thomas Gallou\"et, Michel Ledoux and Ivan Nourdin for their advice and comments, and the anonymous referee for remarks that helped improve this manuscript.

\end{document}